\documentclass[a4paper,10pt]{article}
\title{On the separation of regularity properties of the reals}
\author{Giorgio Laguzzi}

\usepackage{fontenc}
\usepackage{amsmath}
\usepackage{amsfonts}
\usepackage{amssymb}
\usepackage{amsthm}
\usepackage{graphicx}
\usepackage[english]{babel}

\newtheorem{definition2}{Definition}
\newtheorem{fact}[definition2]{Fact}
\newtheorem{lemma}[definition2]{Lemma}

\newtheorem{remark2}[definition2]{Remark}
\newtheorem{theorem}[definition2]{Theorem}

\newtheorem*{Mobservation2}{Main Observation}
\newtheorem{example2}[definition2]{Example}
\newtheorem{corollary}[definition2]{Corollary}

\newenvironment{definition}{\begin{definition2} \upshape}{\end{definition2}}
\newenvironment{remark}{\begin{remark2} \upshape}{\end{remark2}}

\newcommand{\algebra}{\textsc}
\newcommand{\algebraQ}{\algebra{B}}
\newcommand{\amal}{\textsf{Am}}
\newcommand{\asilver}{\poset{VT}}

\newcommand{\baire}{\omega^\omega}
\newcommand{\bairefin}{\omega^{< \omega}}

\newcommand{\bp}{\textbf{all}(\textsc{Baire})}

\newcommand{\cantor}{2^\omega}
\newcommand{\cantorfin}{2^{< \omega}}
\newcommand{\cohen}{\poset{C}}
\newcommand{\conc}{\smallfrown}

\newcommand{\DDelta}{\mathbf{\Delta}}

\newcommand{\dominating}{\poset{D}}
\newcommand{\enfa}{\textit}

\newcommand{\ifif}{\Leftrightarrow}
\newcommand{\force}{\Vdash}

\newcommand{\height}{\text{ht}}
\newcommand{\laver}{\poset{L}}

\newcommand{\lm}{\textbf{all}(\textsc{Lebesgue}) }

\newcommand{\miller}{\poset{M}}

\newcommand{\mm}{\textbf{all}(\textsc{Miller})}
\newcommand{\model}{\textsc}

\newcommand{\poset}{\mathbf}

\newcommand{\random}{\poset{B}}
\newcommand{\real}{\baire}

\newcommand{\restric}{{\upharpoonright}}

\newcommand{\sacks}{\poset{S}}

\newcommand{\silver}{\poset{V}}
\newcommand{\sm}{\textbf{all}(\textsc{Sacks})}
\newcommand{\SSigma}{\mathbf{\Sigma}}

\newcommand{\splitting}{\textsc{Split}}

\newcommand{\stem}{\textsc{Stem}}
\newcommand{\successor}{\textsc{Succ}}

\newcommand{\term}{\textsc{Term}}

\newcommand{\vm}{\textbf{all}(\textsc{Silver})}

\newcommand{\0}{\mathbf{0}}

\begin{document}

\maketitle
\begin{abstract}
% "keep on the" -> "continue" 
% If you don't ike the repetition of "continue" try something like "This paper contributes to a line of research started by .."
We present a model where $\omega_1$ is inaccessible by reals, Silver measurability holds for all sets but Miller and Lebesgue measurability fail for some sets. This contributes to a line of research started by Shelah in the 1980s and more recently continued by Schrittesser and Friedman (see \cite{S10}), regarding the separation of different notions of regularity properties of the real line.
\end{abstract}
% Also, it is good practice to quickly say what you did in the abstract, so switch the order of the two sentences: First "We present a model". Then "This contributes to a line of research".

\paragraph{Acknowledgement.} This work constitutes a part of my doctoral dissertation, written  at the Kurt G\"odel Research Center, University of Vienna, under the supervision of Sy Friedman. I am extremely in debt with him, for many enlightening suggestions which greatly improved the writing of the essay. Furthermore I would like to thank the FWF (Austrian Science Fund) for its indispensable support through the research project \#P22430-N13.
\section{Introduction}
The Lebesgue measurability and the Baire property are certainly the most common notions of regularity of the reals. 
The following results concerning the 2nd level of the projective hierarchy are nowadays part of  the folklore.
\begin{theorem}[Solovay,\cite{So70}]\ \label{solovay-char}
\begin{enumerate}
 \item[(i)] $\SSigma^1_2(\textsc{Lebesgue})$ iff $\forall x \in \baire (\text{random reals over $\model{L}[x]$ form a co-null set})$;
 \item[(ii)] $\SSigma^1_2(\textsc{Baire})$ iff  $\forall x \in \baire (\text{Cohen reals over $\cohen(\model{L}[x])$ form a comeager set})$.
\end{enumerate}
\end{theorem}
\begin{theorem}[Shelah-Judah,\cite{SJ89}]\ \label{shelah-char}
\begin{enumerate}
 \item[(i)] $\DDelta^1_2(\textsc{Lebesgue})$ iff  $\forall x \in \baire \exists z \in \cantor ( z \text{ random over $\model{L}[x]$})$;
 \item[(ii)] $\DDelta^1_2(\textsc{Baire})$ iff  $\forall x \in \baire \exists z \in \cantor ( z \text{ Cohen over $\model{L}[x]$})$.
\end{enumerate}
\end{theorem}

Other important notions of regularity, which have been more recently studied by Brendle, L\"{o}we and Halbeisen, are Sacks-, Miller-, Silver- and Laver-measurability
(we will denote such forcings with the usual letters $\sacks$, $\miller$, $\silver$, $\laver$, and they will be recalled in the next section).
\begin{definition} \label{def:P-measurable}
Let $\poset{P}$ be one among $\sacks$, $\miller$, $\silver$, $\laver$. A set of reals $X$ is said to be $\poset{P}$-measurable iff
\[
 \forall T \in \poset{P} \exists T' \leq T, T' \in \poset{P} ([T'] \subseteq X \vee [T'] \cap X = \emptyset).
\]
We will often refer to such notions by saying Silver measurable, Miller measurable and so on.
\end{definition}
A very detailed work concerning a general approach to these notions of regularity may be found in \cite{K12}, chapter 2, and \cite{Ik10}.
For these properties, one can prove characterizations like in theorem \ref{solovay-char} and \ref{shelah-char}. The following results are due to Brendle, L\"{o}we and Halbeisen.
\begin{theorem}[Brendle-L\"{o}we, \cite{BL99} and Brendle-L\"{o}we-Halbeisen, \cite{BLH05}]\ \label{brendlelowe-char}
\begin{itemize}
 \item[(i)] $\DDelta^1_2(\textsc{Sacks})$ iff $\SSigma^1_2(\textsc{Sacks})$ iff $\forall x \in \baire (\baire \cap \model{L}[x] \neq \baire)$;
 \item[(ii)] $\DDelta^1_2(\textsc{Miller})$ iff $\SSigma^1_2(\textsc{Miller})$ iff $\forall x \in \baire (\baire \cap \model{L}[x] \text{ is not dominating})$;
 \item[(iii)] $\DDelta^1_2(\textsc{Laver})$ iff $\SSigma^1_2(\textsc{Laver})$ iff $\forall x \in \baire (\baire \cap \model{L}[x] \text{ is bounded})$;
 \item[(iv)] $\DDelta^1_2(\textsc{Silver})$ implies $\forall x \in \baire \exists z \in \cantor (z \text{ is splitting over } \model{L}[x])$ \\
$\SSigma^1_2(\textsc{Silver})$ implies  $\forall x \in \baire (\baire \cap \model{L}[x] \text{ is not dominating})$.
\end{itemize}
\end{theorem}

In this paper we focus on Silver, Miller and Lebesgue measurability. In particular from (ii) and (iv) it follows $\SSigma^1_2(\textsc{Silver}) \Rightarrow \SSigma^1_2(\textsc{Miller})$. This implication has partially inspired the result of this paper, which we will present in section
\ref{silver-miller}, where we will construct a model
\[
\model{N}^* \models \vm \wedge \neg \mm \wedge \neg \lm \wedge \forall x \in \baire (\omega_1^{\model{L}[x]} < \omega_1),
\]
showing in particular that the above implication occurring for $\SSigma^1_2$ sets does not shift to the family of all sets of reals. The study of the behaviour of regularity properties on \emph{large} families of subsets of reals was initiated by Solovay (\cite{So70}), and then continued along the years by Shelah (\cite{Sh84} and \cite{Sh85}), Friedman and Schrittesser (\cite{S10}). Another reason which inspired the work was to find a way to drop Lebesgue measurability by iterating Shelah's amalgamation and obtaining $\omega_1$ inaccessible by reals.
Note that, in the Cohen model (i.e., the extension via adding $\omega_1$-many Cohen reals), Silver measurability holds for all projective sets, but it is unclear how the Miller measurability behaves in this model. Moreover, our purpose was also to have $\omega_1$ inaccessible by reals, so Cohen model was not of interest from this point of view
(in particular we wanted full-regularity on $\SSigma^1_2$).

We conclude this introductory section with a schema of the article: in section \ref{preliminaries} we review basic concepts and notation that we use throughout the paper; in 
section \ref{amalgamation} we introduce the tools we use later on: Shelah's amalgamation, unreachability and amoeba for Silver; then, section \ref{silver-miller} is devoted to prove the main result mentioned above. A last section is finally devoted to some concluding comments and possible further developments of the investigation.

\section{Preliminaries} \label{preliminaries}
Our notation is rather standard. A tree $T$ is a subset of $\cantorfin$ or $\bairefin$ closed under initial segments, i.e., for every $t \in T$, $t \restric k \in T$,
for every $k < |t|$, where $|t|$ represents the length of $t$.
We denote with $\stem(T)$ the longest element $t \in T$ compatible with every node of $T$, and we set $k \in \successor(t,T)$ iff $t^\conc k \in T$. We use the notation $t \unlhd t'$ meaning that $t$ is an initial segment of $t'$.
For every $t \in T$, we say that $t$ is a splitting node whenever $|\successor(t,T)| \geq 2$, and we denote with $\splitting(T)$ the set of all splitting nodes.
Moreover, for $n \geq 1$, we say $t \in T$ is an $n$th splitting node iff  $t \in \splitting(T)$ and $n$ is maximal such that there are $k_0 < \dots < k_{n-1}=|t|$ natural numbers such that
$t \restric k_j \in \splitting(T)$,
for every $j \leq n-1$, and we denote with $\splitting_n(T)$ the set consisting of the $n$th splitting nodes (note that, under this notation, $\stem(T)$ is the
$1$st splitting node, with $k_0=|\stem(T)|$). Furthermore, for every $t \in T$, the set $\{s \in T: s \text{ is compatible with }t \}$ is denoted with $T_t$. When $T$ is a
finite tree, $\term(T)$ denotes the set consisting of those $t$'s having no extension in $T$, and $\height(T):= \max\{n: \exists t \in T, |t|=n \}$ represents the height of $T$.
Finally, the \emph{body of $T$} is defined by $[T]:= \{ x: \forall n (x \restric n \in T)   \}$ and we say that $F \subseteq T$ is a \emph{front} iff $F$ is an antichain and for every $x \in [T]$ there exists $t \in F$ such that $t \unlhd x$. 

In this paper we deal with \enfa{Sacks} (or \enfa{perfect}) trees, i.e.,  trees such that each node can be extended to a splitting node. In particular, we focus
the attention on some particular types of perfect trees:
\begin{itemize}
 \item $T \subseteq \cantorfin$ is a \enfa{Silver tree} (or \enfa{uniform tree}) iff $T$ is perfect and for every $s,t \in T$, such that $|s|=|t|$, one has
$s^\conc 0 \in T \ifif t^\conc 0 \in T$ and $s^\conc 1 \in T \ifif t^\conc 1 \in T$.
\item $T \subseteq \bairefin$ is a \enfa{Miller tree} (or \enfa{superperfect tree}) iff $T$ is perfect and for every $t \in \splitting(T)$, one has
$|\successor(t,T)|=\omega$.
\end{itemize}
Clearly, Silver (Miller) forcing is the poset consisting of Silver (Miller) trees  ordered by inclusion. We denote these forcings by $\silver$ and $\miller$, respectively. Furthermore, if $G$ is the $\silver$-generic filter over the ground model $\model{N}$, we call the generic
branch $z_G= \bigcup\{ \stem(T): T \in G \}$ a Silver real (and analogously for Miller).
Other notions of forcing relevant to the topic of seperating regularities, which can similarly be presented in terms of their associated trees include: Sacks, Laver (mentioned in the introduction), Mathias, Heckler, eventually different and Matet, but we will not deal with them in this paper. 
Other posets which we will use throughout the paper will be the Cohen forcing $\cohen$, consisting of finite sequences of $0$s and $1$s, ordered by extension,
and the random forcing $\random$, consisting of perfect trees $T$ such that for every $t \in T$, $\mu([T_t])>0$, ordered by inclusion.

For the purpose of this paper we can use the following result as a definition of $\poset{P}$-measurability for a \emph{topologically reasonable pointclass} $\Theta$, i.e., a family of sets closed under continuous preimages and intersections with closed sets. 
\begin{lemma}[Brendle-L\"{o}we, \cite{BL99}, lemma 2.1] \label{lemma:theta-measurability}
Let $\poset{P} \in \{\silver, \miller \}$ and let $\Theta$ be a topologically reasonable family of sets of reals. Then
$\Theta(\poset{P}) \equiv$ ``every set in $\Theta$ is $\poset{P}$-measurable'' iff
\[
 \forall X \in \Theta, \quad \exists T \in \poset{P} ([T] \subseteq X \vee [T] \cap X= \emptyset).
\]
\end{lemma}
Note that the family of projective sets, the family of $\SSigma^1_n$-, $\DDelta^1_n$-, $\mathbf{\Pi}^1_n$-sets are topologically reasonable. Finally, also the family of all sets of reals, which we denote by $\textbf{all}$, is trivially topologically reasonable.

We now prove a simple result, which gives us a direct implication between Baire property and Silver measurability.
\begin{fact} \label{fact:comeager-silver}
Any comeager set contains the body of a Silver tree.
\end{fact}
\begin{proof}
Let $Y \supseteq \bigcap_{n \in \omega} D_n$, where all $D_n$'s are open dense. We use the following notation: for every $s,t \in \cantorfin$, put
\[
t \oplus s := \{t' \in \cantorfin: \forall n < |t|(t'(n)=t(n)) \wedge \forall n \geq |t|( t'(n)= s(n))\}.
\]
Consider the following recursive construction:
\begin{itemize}
 \item let $t_\emptyset \in \cantorfin$ such that $[t_\emptyset] \subseteq D_0$;
 \item Assume for every $r \in 2^{n}$ we have already defined $t_r$ such that $[t_r] \subseteq D_n$.
 Let $\{ t_j: j < 2^{n+1} \}$ be an enumaration of $\{ {t_r}^{\conc}i: r \in 2^n, i=0,1  \}$. Then
consider the following construction along $j < 2^{n+1}$:
\begin{itemize}
\item[]for $j=0$, pick $s^0 \trianglerighteq t_0$ such that $[s^0] \subseteq D_{n+1}$;
\item[]for $j+1$, pick $s^{j+1} \trianglerighteq  t_{j+1} \oplus s^j$ such that $[s^{j+1}] \subseteq D_{n+1}$.
\item[] Then, for every $r \in 2^n$ and $i=0,1$, put $t_{r^{\conc} i} = t_j \oplus s^{2^{n+1}-1}$, where $t_j= {t_{r}}^{\conc}i$.
\end{itemize}
\end{itemize}
Finally, put $R:= \{ t_{r}: r \in \cantorfin, t_{r} \text{ as defined in the construction} \}$ and $T:= \{ t \in \cantorfin: \exists t' \in R \exists k \leq |t'| (t' \restric k = t)  \}$ (i.e., $T$ is the downward closure of $R$). It is clear that $T$ is a Silver tree such that
for every $z \in [T]$, $z \in \bigcap_{n \in \omega} D_n$.
\end{proof}

\begin{corollary} \label{corollary:cohen-silver}
If $\Theta$ is a topologically reasonable family, then $\Theta(\textsc{Baire}) \Rightarrow \Theta(\textsc{Silver})$.
\end{corollary}
\begin{proof}
Pick a set $X \in \Theta$ having the Baire property. Then, if $X$ is not meager, there exists $s \in \cantor$ such that $X$ is comeager in $[s]$. Hence, by fact \ref{fact:comeager-silver},
one can find $T \in \silver$, with $\stem(T)=s$, such that $[T] \subseteq X$. In case $X$ is meager, one can analogously find a Silver tree contained in the complement of $X$.
By lemma \ref{lemma:theta-measurability}, that is sufficient to obtain $\Theta(\textsc{Silver})$.
\end{proof}

\section{Shelah's amalgamation and unreachability} \label{amalgamation}
In the construction of the model about Lebesgue measurability and Baire pro\-perty (see \cite{So70}), Solovay used a key property of the L\'{e}vy-collapsing algebra,
which we recall in the following definition. 

\begin{definition} \index{reflection property}
A complete Boolean algebra $\algebraQ$ has the \enfa{Solovay property} if and only if for any formula $\Phi$ with parameters in the ground model $\model{N}$
and for any $\algebraQ$-name for a real $\dot{x}$, one has $\Vert \Phi(\dot{x}) \Vert_\algebraQ \in \algebraQ_{\dot{x}}$, where $\algebraQ_{\dot{x}}$ is the complete
Boolean algebra generated by $\dot{x}$, i.e., $\algebraQ_{\dot{x}}$ is generated by $\{ \Vert s \lhd \dot{x} \Vert_\algebraQ:  s \in \cantorfin \}$.
\end{definition}
The meaning of the definition is that, to evaluate $\Phi(\dot{x})$ in $\model{N}^\algebraQ$, it suffices to know its value in a certain partial extension
obtained from a subalgebra of $\algebraQ$, namely $\algebraQ_{\dot{x}}$.
It is not hard to show that a particular family of complete Boolean algebras, satisfying the Solovay property, is the class of strongly homogeneous algebras,
which we now define.
\begin{definition}
A complete Boolean algebra $\algebraQ$ is \enfa{strongly homogeneous} \index{strongly homogenous algebra} if and only if for every pair of $\sigma$-generated complete subalgebras
$\algebraQ_1,\algebraQ_2 \lessdot \algebraQ$, every isomorphism $\phi^*: \algebraQ_1 \rightarrow \algebraQ_2$ can be extended 
to an automorphism $\phi: \algebraQ \rightarrow \algebraQ$.
\end{definition}
\noindent Lemma 9.8.3 in \cite{BJ95} shows that
\begin{equation}
\text{if $\algebraQ$ is strongly homogeneous, then $\algebraQ$ satisfies the Solovay property}.
\end{equation}
Note that the L\'{e}vy-collapsing algebra is strongly homogeneous.
In \cite{Sh84}, Shelah introduced \enfa{amalgamation}, a general method to build Boolean algebras satisfying a property related to strong homogeneity, and in effect providing us with a tool to prove variations of Solovay's result.

\paragraph{Shelah's construction.} We review Shelah's amalgamation in as much detail as we need for the present purpose and refer the reader to the splendid exposition in \cite{JR93} for details.

\begin{definition}
Let $\algebra{B}$ be a complete Boolean algebra and $\algebra{B}_0 \lessdot \algebra{B}$. The \emph{projection} map $\pi: \algebra{B} \rightarrow \algebra{B}_0$ is defined by $\pi(b)=  \prod \{b \leq b_0: b_0 \in \algebra{B}_0 \}$.
\end{definition}
\begin{definition}
Let $\algebra{B}$ be a complete Boolean algebra and $\algebraQ_1, \algebraQ_2$ two isomorphic complete subalgebras of $\algebra{B}$ and $\phi_0$ the isomorphism between them.
One defines the \enfa{amalgamation of $\algebra{B}$ over $\phi_0$}, say $\amal(\algebra{B}, \phi_0)$, as follows: first, let
\[
\algebra{B} \times_{\phi_0} \algebra{B} := \{ (b',b'') \in \algebra{B} \times \algebra{B}: \phi_0(\pi_1(b')) \cdot \pi_2(b'')\neq \mathbf{0} \},
\]
where $\pi_j: \algebra{B} \rightarrow \algebra{B}_j$ is the projection, for $j=1,2$, and consider on such $\algebra{B} \times_{\phi_0} \algebra{B}$ simply
the product order. Then set \index{amalgamation $\amal$} $\amal(\algebra{B}, \phi_0):= B(\algebra{B} \times_{\phi_0} \algebra{B})$, i.e., the complete Boolean algebra generated by
$\algebra{B} \times_{\phi_0} \algebra{B}$.
\end{definition}
One can easily see that $e_j: \algebra{B} \rightarrow \amal(\algebra{B}, \phi_0)$ such that
\[
 e_1(b)= (b,\mathbf{1}) \text{ and } e_2(b)=(\mathbf{1},b)
\]
are both complete embeddings (\cite{JR93}, lemma 3.1). Further, for any $b_1 \in \algebra{B}_1$, one can show that
\begin{equation} \label{equa:amalgamation}
 (b_1, \mathbf{1}) \text{ is equivalent to } (\mathbf{1}, \phi_0(b_1)).
\end{equation}
In fact, assume $(a',a'') \leq (b_1, \mathbf{1})$ and $(a',a'')$ incompatible with $(\mathbf{1}, \phi_0(b_1))$ (in $\amal(\algebra{B}, \phi_0)$). The former implies $\pi_1(a') \leq b_1$, while the latter implies $\pi_2(a'') \cdot \phi_0(b_1)=\mathbf{0}$, and hence one obtains $\phi_0(\pi_1(a'))\cdot \pi_2(a'')=\mathbf{0}$, which means that the pair $(a',a'')$ does not belong to the amalgamation.

Moreover, if one considers $f_1: e_2[\algebra{B}] \rightarrow e_1[\algebra{B}]$ such that, for every
$b \in \algebra{B}$, $f_1(\mathbf{1},b)=(b,\mathbf{1})$, one obtains an isomorphism between two copies of
$\algebra{B}$ into $\amal(\algebra{B}, \phi_0)$, such that $f_1$ is an extension of $\phi_0$ (since for every
$b_1 \in \algebra{B}_1$, by (\ref{equa:amalgamation}) above, $e_1(b_1)= (b_1, \mathbf{1})=(\mathbf{1}, \phi_0(b_1))=e_2(\phi_0(b_1))$, which means $e_1{\upharpoonright} B_1 = e_2\circ \phi_0$).

Hence, if one considers $e_1[\algebra{B}], e_2[\algebra{B}]$ as two isomorphic complete subalgebras of $\amal(\algebra{B}, \phi_0)$, one can repeat the same
procedure to construct
\[
2\text{-}\amal(\algebra{B}, \phi_0) := \amal(\amal(\algebra{B}, \phi_0), f_1)
\]
and $f_2$ the isomorphism between two copies of $\amal(\algebra{B}, \phi_0)$ extending
$f_1$. 
It is clear that one can continue such a construction, in order to define, for every $n \in \omega$,
\[
n+1\text{-}\amal(\algebra{B}, \phi_0) :=\amal(n\text{-}\amal(\algebra{B}, \phi_0), f_n)
\]
and $f_{n+1}$ the isomorphism between two copies of $n\text{-}\amal(\algebra{B}, \phi_0)$ extending $f_n$.

Finally, putting
\begin{itemize}
\item[(i)] $\omega\text{-}\amal(\algebra{B}, \phi_0)=$ Boolean completion of direct limit of $n\text{-}\amal(\algebra{B}, \phi_0)$'s, and
\item[(ii)] $\phi = \lim_{n \in \omega} f_n$ (in the obvious sense),
\end{itemize}
one obtains $\algebra{B}_1, \algebra{B}_2 \lessdot \omega\text{-}\amal(\algebra{B}, \phi)$ and $\phi$ automorphism of
$\omega\text{-}\amal(\algebra{B}, \phi_0)$ extending $\phi_0$.

We shall abuse terminology by referring to the \emph{Boolean completion} of the direct limit of a sequence of Boolean algebras simply as their \emph{direct limit} (since only complete Boolean algebras are of interest to us). We write $\lim_{\alpha < \lambda} \algebra{B}_\alpha$ for the direct limit understood in this way.

We will iterate this construction (each time with a new pair of isomorphic sub-algebras) as a method to obtain a Boolean algebra which satisfies a particular variant of strong homogeneity.

\paragraph{Unreachability.} A crucial ingredient to the present result is \emph{unreachability}, a property of reals which in a sense is preserved both by Silver forcing and by amalgamation. A real is unreachable if it avoids every \emph{slalom} of the ground model. 

\begin{itemize}
\item $\Gamma_k = \{ \sigma \in \mathbf{HF}^\omega: \forall n \in (|\sigma(n)| \leq 2^{kn})   \}  \}$ and $\Gamma= \bigcup_{k \in \omega} \Gamma_k$, where $\mathbf{HF}$ denotes the hereditary finite sets;
\item let $g(n)=2^n$ and $\{ I_n: n \in \omega \}$ be the partition  of $\omega$ such that $I_0= \{ 0 \}$ and $I_{n+1}= \Big [\sum_{j \leq n} g(j), \sum_{j \leq n+1} g(j) \Big )$, for every $n \in \omega$;
\item given $x \in \cantor$, define $h_x(n)= x \restric I_n$.
\end{itemize}
\begin{definition}
One says that $z \in \cantor$ is \enfa{unreachable over $\model{N}$} iff
\[
\forall \sigma \in \Gamma \cap \model{N} \exists n \in \omega (h_z(n) \notin \sigma(n)).
\]
\end{definition}
\begin{remark} \label{remark:random-unreachable}
If $z$ is random over $\model{N}$, then $z$ is unreachable over $\model{N}$.
To prove that, assume towards a contradiction that there is $\sigma \in \Gamma_k \cap \model{N}$ such that for every $n \in \omega$, $h_z(n) \in \sigma(n)$. Consider the set $B:= \{ x \in \cantor: \forall n \in \omega (h_x(n) \in \sigma(n)) \}$, which is in $\model{N}$ by construction. Since
$\lim_{n \in \omega} \frac{2^{kn}}{2^{g(n)}}=0$, we get that $B$ has measure zero. Hence $z \notin B$, which is a contradiction.
\end{remark}
\begin{remark} \label{remark:cohen-unreachable}
If $x$ is Cohen over $\model{N}$ then $x$ is unreachable over $\model{N}$. The proof is similar to the above one, by noting that the set $B$ is closed nowhere dense too, simply bacause for every $n \in \omega$, $|\sigma(n)| < 2^{g(n)}$ and so, given a sequence $s \in \cantorfin$, one can find an extension $s'$ such that $[s'] \cap B = \emptyset$.
\end{remark}
\begin{lemma} \label{lemma:preserving-unreachability}
Let $\algebra{B}, \algebra{B}_1, \algebra{B}_2, \phi_0, e_1, e_2$ as above and $\dot{x}$ a $\algebra{B}$-name for an element of $\cantor$. If
$\force_{\algebra{B}} \text{`` $\dot{x}$ is unreachable over $\model{N}^{\algebra{B}_1}$ and  $\model{N}^{\algebra{B}_2}$ ''},$
then
\[
\force_{\amal(\algebra{B},\phi_0}) \text{`` $e_1(\dot{x})$ is unreachable over $\model{N}^{e_2[\algebra{B}]}$ ''},
\]
and analogously $\force_{\amal(\algebra{B},\phi_0}) \text{`` $e_2(\dot{x})$ is unreachable over $\model{N}^{e_1[\algebra{B}]}$ ''}$.
\end{lemma}
\begin{proof}
It is proven in \cite{JR93} that
\[
\amal(\algebra{B}, \phi_0)/\algebra{B}_1 \text{ densely embeds into }  e_1[\algebra{B}/\algebra{B}_1] \times e_2[\algebra{B}/\algebra{B}_2].
\]
That roughly means that, in $\model{N}^{\algebra{B}_1}$, the amalgamation (quotiented by $\algebra{B}_1$) can be seen as a product of the two copies of $\algebra{B}$. Hence, for our proof it suffices to show that, if $\algebra{A}_0$ and $\algebra{A}_1$ are two complete Boolean algebras and $\dot{x}$ is an $\algebra{A}_0$-name for an element in $\cantor$ such that $\force_{\algebra{A}_0}$ `` $\dot{x}$ is unreachable over $\model{N}$'', then
\[
\force_{\algebra{A}_0 \times \algebra{A}_1} \text{`` $\dot{x}$ is unreachable over $\model{N}[G]$ ''},
\]
where $G$ is $\algebra{A}_1$-generic over $\model{N}$. In fact by considering $\algebra{A}_0$ and $\algebra{A}_1$ to be $e_1[\algebra{B}/\algebra{B}_1]$ and $e_2[\algebra{B}/\algebra{B}_2]$, respectively (and $\model{N}^{\algebra{B}_1}$ as ground model), we obtain exactly the conclusion of the lemma.

To reach a contradiction, assume there is $\sigma \in \Gamma_k \cap \model{N}[G]$ and $(a_0,a_1) \in \algebra{A}_0 \times \algebra{A}_1$ such that $(a_0,a_1) \force \forall n \in \omega (h_x(n) \in \sigma(n))$. For each $n \in \omega$ one can pick $b_n \in \algebra{A}_1$, $b_n \leq a_1$ and $ W_n \subset \omega$, with $|W_n| \leq 2^{kn}$, such that $b_n \force \sigma(n)=W_n$. Furthermore, since the sequence $\langle W_n: n \in \omega  \rangle $ is in $\model{N}$, one can find $a \in \algebra{A}_0$, $a \leq a_0$ and $j \in \omega$ such that $a \force h_x(j) \notin W_j$. Hence, we would get, on the one hand $(a,b_j) \leq (a_0, a_1)$ and so $(a, b_j) \force \forall n \in \omega (h_x(n) \in \sigma(n))$, but on the other hand $(a,b_j) \force h_x(j) \notin W_j=\sigma(j)$.
\end{proof}
\begin{lemma} \label{lemma:silver-preserves-unreachable}
Assume $x \in \cantor$ be unreachable over $\model{N}$. Then $x$ remains unreachable over $\model{N}[z]$, where $z$ is a Silver real. In other words, the property of being unreachable is preserved by Silver extensions.
\end{lemma}
\begin{proof}
It is a standard fusion argument. Given $\sigma \in \Gamma_k \cap \model{N}[z]$ and a condition $T \in \silver$, the idea is to construct a fusion sequence $\langle T_n: n \in \omega \rangle$ and $\tau \in \Gamma_{k+1} \cap \model{N}$ such that the limit of the fusion $T' \force \forall n \in \omega (\sigma(n) \subseteq \tau(n))$. The key-point of the proof is that for every $n \in \omega$, one only has $2^{n}$-many $n+1$st splitting nodes, and therefore $2^{n}$-many possible decisions for $\sigma(n)$. In this way, one can define $\tau(n)$ to be the union of all these possibilities, in order to have $|\tau(n)| \leq 2^{k  n} \cdot 2^{n}= 2^{(k+1)n}$, which gives us $\tau \in \Gamma_{k+1} \cap \model{N}$.

More formally: \textsc{Step 0}: Let $T^0 \leq T$ such that $T^0 \force \sigma(0) = \tau(0)$, for some sigleton $\tau(0) \subset \omega$. \textsc{Step $n+1$}: Let $\{ t_j : j < 2^{n+1} \}$ be an enumeration of the set $\{ {t_m}^\conc i: t_m \in \splitting_{n+1}(T^{n}) \land i=0,1 \}$. For any $j < 2^{n+1}$, one can find $T^{n+1}_j \leq T^{n}_{t_j}$ and $E^{n+1}_j$ of size $\leq 2^{k(n+1)}$ such that $T^{n+1}_j \force \sigma(n+1)= E^{n+1}_j$. Note also that by using an argument as in the proof of lemma \ref{lemma:amoeba-silver} one can uniformly pick those $T^{n+1}_j$'s, in order to obtain a Silver tree $T^{n+1} := \bigcup \{ T^{n+1}_j: j < 2^{n+1} \}$. If we now put $\tau(n+1)= \bigcup \{ E^{n+1}_j: j < 2^{n+1} \}$ we then get $T^{n+1} \leq_{n} T^{n}$, $|\tau(n+1)| \leq 2^{n+1} \cdot 2^{k(n+1)}= 2^{(k+1)(n+1)}$ and $T^{n+1} \force \sigma(n+1) \subseteq \tau(n+1)$ (where $S \leq_n T$ means $S \leq T$ and $\splitting_{n+1}(S)=\splitting_{n+1}(T)$.)

Finally put $T'= \bigcap_{n \in \omega} T_n$, for every $n \in \omega$. Hence $\tau \in \Gamma_{k+1} \cap \model{N}$, $T' \leq T$ and $T' \force \forall n \in \omega (\sigma(n) \subseteq \tau(n))$.

\end{proof}

\paragraph{Amoeba for Silver.}
Proofs involving regularity properties need the right notion of amoeba, i.e., a particular forcing notion to add a \emph{large} set of \emph{generic} reals, where the precise meaning of the two italic-style words depend on the notion of regularity we are dealing with.

\begin{definition}
\[
\asilver= \{ (p, T): T \in \silver \text{ and } p= T \restric n, \text{ for some } n \in \omega  \},
\]
ordered by
\[
(p',T') \leq (p,T) \Leftrightarrow T' \subseteq T \land  T' \restric \height(p)= T \restric \height(p).
\]
\end{definition}
We aim at showing that this forcing adds \emph{many} Silver reals, more precisely,
\begin{equation} \label{eq:VTgeneric}
 \force_{\asilver} \forall T \in \silver \cap \model{N} \exists T' \subseteq T (T' \in \silver \land [T'] \subseteq \silver(\model{N})),
\end{equation}
where we remind that $\silver(\model{N})$ denotes the set of Silver reals over the ground model $\model{N}$.

First of all, we prove the following preliminary fact.
\begin{lemma} \label{lemma:amoeba-silver}
 Let $T_G = \bigcup \{ p : \exists T((p, T) \in G) \}$, where $G$ is $\asilver$-generic over the ground model. Then  $\model{N}[G] \models  [T_G] \subseteq \silver(\model{N})$.
\end{lemma}
\begin{proof}
Fix an open dense $D \subseteq \silver$ and $(p,T) \in \asilver$.
First of all, let $t_0, t_1, \dots, t_k$ be an enumeration of all terminal nodes in $p$. We use the following notation: for any tree $T$ and
$t \in \cantorfin$ such that $|t| \leq |\stem(T)|$,
\[
t \oplus T := \{t \oplus t': t' \in T\}, \index{$\oplus$}
\]
i.e., the tree obtained from $T$ by chancing all nodes to begin as $t$. We aim at uniformly shrinking $T$ to some $T' \in \silver$ so that $(p,T') \force \forall z \in [T_G](H_z \cap D \neq \emptyset)$, where $H_z$ is defined by $H_z= \{ S \in \silver \cap \model{N}: z \in [S]  \}$.
Consider the following construction: \label{fact:silver.silver}
\begin{itemize}
 \item firstly, pick $T^0_{t_0} \subseteq T_{t_0}$ in $D$ and let $T^0_{t_1}= t_1 \oplus T^0_{t_0}$;
 \item then, pick $T^1_{t_1} \subseteq T^0_{t_1}$ in $D$ and let $T^1_{t_2}= t_2 \oplus T^1_{t_1}$; note that $t_0 \oplus T^1_{t_1} \subseteq T^0_{t_0}$
 and so $t_0 \oplus T^1_{t_1} \in D$ as well;
 \item continue this construction for every $j \leq k$ and finally let $T'_{t_j}=t_j \oplus  T^k_{t_k}$, for every $j \leq k$.
\end{itemize}
It follows from the construction that $T':= \bigcup \{ T'_{t_j} : j \leq k \}$ is a Silver tree and, for any $z \in [T_G]$, one has $H_z \cap D \ni T'_{t_j} $, for the appropriate $j \leq k$ such that $t_j \lhd z$. Hence, we have shown that for every branch $z \in [T_G]$, $H_z \cap D \neq \emptyset$.

It is left to show that the set $H_z$ is a filter. Pick $T_1, T_2 \in H_z$ incompatible (note that by absoluteness they are incompatible in $\model{N}$ as well). Hence, $[T_1] \cap [T_2]$ is finite, i.e., $[T_1] \cap [T_2]= \{ x_i: i \leq n \}$. Then $E:= \{ T \in \silver: \forall i \leq n (x_i \notin [T])  \}$ is open dense set in the ground model \model{N}, and so, by genericity, there is $T \in E$ such that $z \in [T]$, which contradicts $T_1, T_2 \in H_z$ (and so $z \in [T_1] \cap [T_2]$). 
\textbf{N.B.}: by absoluteness, this argument works when $z$ belongs not only to $\model{N}[G]$, but to any ZFC-model  $ \model{M} \supseteq \model{N}[G]$ (see also remark \ref{remark-amoeba-absoluteness} coming). 

\end{proof}

\begin{remark} \label{remark-amoeba-absoluteness}
The forcing we have just introduced is an amoeba in a strong sense, which means that the tree added by $\asilver$ is a Silver tree of Silver reals in any ZFC-model $\model{M} \supseteq \model{N}[G]$, where $G$ is $\asilver$-generic over the ground model $\model{N}$. The method for proving that is essentially the same used by Spinas in \cite{Sp95} about an analogous result for an amoeba of Laver. In fact, if we look at the proof of \ref{lemma:amoeba-silver}, we actually show that, for every open dense set $D \subseteq \silver$ of the ground model there exists a front $F \subseteq T_G$ such that for every $t \in F$, $(T_G)_t \in D$. Since  being a front is a $\Pi^1_1$-property, by absoluteness, it exists in any ZFC-model $\model{M} \supseteq \model{N}[G]$. It therefore follows that our argument works even if $z \in [T_G]$ comes from any ZFC-model $\model{M} \supseteq \model{N}[G]$. In other words, for any ZFC-model $\model{M} \supseteq \model{N}[G]$, $\model{M} \models [T_G] \subseteq \silver(\model{N})$.
\end{remark}

It is left to show that this forcing actually adds such a generic Silver tree inside any Silver tree of the ground model. To this aim, fix any $S \in \silver \cap \model{N}$, and consider the forcing $\asilver_{S}$ defined as $\asilver_{S} := \{ (p,T) \in \asilver : T \subseteq S \}$, with the analogous order. It is therefore clear that we can similarly show that any branch through the generic $T_G$ added by $\asilver_{S}$ is Silver generic, and obviously $T_G \subseteq S$. Furthermore, by using the standard $\unlhd$-preserving bijection between $S$ and $\cantor$, one can easily note that $\asilver_{S}$ is forcing equivalent to $\asilver$, actually really isomorphic, and therefore we obtain (\ref{eq:VTgeneric}). 

Finally remark \ref{remark-amoeba-absoluteness} gives also the following corollary.
\begin{corollary} \label{absolute-amoeba}
For every ZFC-model $\model{M} \supseteq \model{N}[G]$
\[
\model{M} \models \forall T \in \silver \cap \model{N} \exists T' \subseteq T (T' \in \silver \land [T'] \subseteq \silver(\model{N})).
\]
\end{corollary}

\section{Silver without Miller and Lebesgue} \label{silver-miller}
We now have all needed tools to show the main results of the paper, that is to provide a model 
\[
\model{M} \models \vm \wedge \neg \mm \wedge \neg \lm \wedge \forall x \in \baire (\omega_1^{\model{L}[x]} < \omega_1),
\]
We remark that in our proof the use of unreachbility together with the amoeba of Silver shows a new method for separating regularity properties using Shelah's amalgamation. 

We start with an inaccessible $\kappa$ and force to add a non-Miller measurable set $Y$ and a non-Lebesgue measurable set $Z$, and we simultaneously amalgamate over Silver forcing $\silver$, with respect to such $Y$ and $Z$.
The construction will give us a complete Boolean algebra $\algebraQ_\kappa$ forcing  \footnote{Note that, following Solovay's approach, we could equivalently pick $\model{HOD}(\text{On}^\omega,Y,Z)$ as inner model, i.e., the class of all sets that are hereditarily ordinal definable over $\text{On}^\omega \cup \{ Y \} \cup \{ Z \}$.} 
\begin{gather*}
\text{``every set of reals in $\model{L}(\baire, Y,Z)$  is Silver measurable,}  \\
\ \ \text{$Y$ is not Miller measurable, $Z$ is not Lebesgue measurable, and} \\
\ \ \text{$\omega_1$ is inaccessible by reals ''}.
\end{gather*}

Observe that we must show not only all sets in $\model{L}(\baire)$
are regular, but all sets in $\model{L}(\baire,Y,Z)$.
Furthermore, intuitively, since we want Silver measurability but \emph{not} Lebesgue and Miller measurability, one should ask a type of homogeneity involving Silver subalgebras of $\algebra{B}_\kappa$ w.r.t. $\dot Y$ and $\dot Z$, but not all $\sigma$-generated subalgebras (since, e.g., fixing $\dot Y$ by Cohen homogeneity would affect the unreachability). More precisely, we want our amalgamation to catch all subalgebras generated by $\algebra{A} \cup b$, where $\algebra{A} \lessdot \algebra{B}_\kappa$ is isomorphic to the Silver algebra and $b \in \algebra{B}_\kappa$. In this spirit, one introduces the following notion. 
\begin{definition}   \label{def:Y.homo}
Let $\algebra{B}$ be a complete Boolean algebra, $\dot{Y}$ and $\dot{Z}$ be $\algebra{B}$-names.
Let $B^+(\silver)$ denote a complete Boolean subalgebra generated by $B(\silver) \cup \{b\}$, for some $b \in \algebra{B}$.
One says that $\algebra{B}$ is $(\silver,\dot{Y},\dot{Z})$-homogeneous if and only if for any isomorphism $\phi_0$
between two complete subalgebras $\algebra{B}', \algebra{B}''$ of $\algebra{B}$,  \index{$(\silver,\dot{Y})$-homogeneous}
such that $\algebra{B}' \approx \algebra{B}'' \approx B^+(\silver)$, there exists
$\phi: \algebra{B} \rightarrow \algebra{B}$ automorphism extending $\phi_0$ such that $\force_{\algebra{B}} \text{`` }\phi (\dot{Y})= \dot{Y} \text{ and } \phi (\dot{Z})= \dot{Z} \text{ ''}$.
\end{definition}

\begin{remark} \label{remark-silver}
Note that for every $b \in \algebra{B}$, one can easily construct a dense embedding between $B(\silver)$ and $B^+(\silver)$, and so they give rise to the same extension.
\end{remark}

So one starts from a ground model $\model{N}$ containing an
inaccessible cardinal $\kappa$. Define a complete Boolean algebra
$\algebraQ_\kappa$ as a direct limit of $\kappa$-many complete Boolean algebras
$\algebraQ_\alpha$'s of size $< \kappa$, such that for every
$\alpha < \gamma < \kappa$, $\algebraQ_\alpha \lessdot
\algebraQ_\gamma$, and one simultaneously constructs two sets
$\dot{Y}$ and $\dot{Z}$ of $\algebraQ_\kappa$-names of reals. We now see in detail such a construction.

\begin{itemize}
 \item Firstly, to obtain the $(\silver,\dot{Y}, \dot{Z})$-homogeneity, we use a standard book-keeping argument as follows:
whenever $\algebra{B}_\alpha \lessdot \algebra{B}' \lessdot \algebra{B}_\kappa$ and $\algebra{B}_\alpha \lessdot \algebra{B}'' \lessdot \algebra{B}_\kappa$ are such that
$\algebra{B}_\alpha$ forces $(\algebra{B}':\algebra{B}_\alpha) \approx (\algebra{B}'':\algebra{B}_\alpha) \approx B^+(\silver)$
and $\phi_0: \algebraQ' \rightarrow \algebraQ''$ an isomorphism s.t. $\phi_0 \upharpoonright \algebra{B}_\alpha= \text{Id}_{\algebra{B}_\alpha}$, then there exists
a sequence of functions in order to extend the isomorphism $\phi_0$ to an automorphism $\phi: \algebra{B}_{\kappa}
\rightarrow \algebra{B}_{\kappa}$, i.e., $\exists \langle \alpha_\eta : \eta
< \kappa \rangle$ increasing, cofinal in $\kappa$, with $\alpha_0=\alpha$, and $ \exists \langle
\phi_\eta : \eta < \kappa \rangle$ such that 
\begin{itemize}
\item for $\eta >0$ successor ordinal, $\algebra{B}_{\alpha_{\eta}+1} =\omega\text{-}\amal(\algebra{B}_{\alpha_{\eta}},\phi_{\eta-1})$,
and $\phi_\eta$ is the automorphism on  $\algebra{B}_{\alpha_\eta+1}$ generated by the amalgamation;
\item  for $\eta$ limit ordinal, let $\algebra{B}_{\alpha_\eta}= \lim_{\xi < \eta} \algebra{B}_{\alpha_\xi}$ and $\phi_\eta= \lim_{\xi < \eta} \phi_\xi$, in the obvious sense;
\item for every $\eta< \kappa$, we have $\algebra{B}_{\alpha_{\eta}+1} \lessdot \algebra{B}_{\alpha_{\eta+1}}$, i.e., $\alpha_\eta+1 < \alpha_{\eta+1}$.
\end{itemize}
Moreover, since one needs to fix the set of names by each automorphism $\phi_\eta$, one puts
\begin{itemize} 
\item successor case $\eta>0$:
\[
\begin{split}
\dot{Y}_{\alpha_{\eta} +1} &:= \dot{Y}_{\alpha_{\eta}} \cup \{
\phi^j_{\eta}(\dot{y}), \phi^{-j}_{\eta}(\dot{y}): \dot{y} \in
\dot{Y}_{\alpha_{\eta}}, j \in \omega \},\\
\dot{Z}_{\alpha_{\eta} +1} &:= \dot{Z}_{\alpha_{\eta}} \cup \{
\phi^j_{\eta}(\dot{z}), \phi^{-j}_{\eta}(\dot{z}): \dot{z} \in
\dot{Z}_{\alpha_{\eta}}, j \in \omega \};\\
\end{split}
\]
\item limit case: $\dot Y_{\alpha_\eta} := \bigcup_{\xi < \eta} \dot Y_{\alpha_\eta}$, and $\dot Z_{\alpha_\eta} := \bigcup_{\xi < \eta} \dot Z_{\alpha_\eta}$. 

\end{itemize}
 \item Secondly, to obtain the Silver measurability of all sets in $\model{L}(\baire,Y,Z)$ together with $Y$ non-Miller measurable and $Z$ non-Lebesgue measurable, one has to add
the following operations into the construction of $\algebra{B}_\kappa$:
\begin{enumerate}
\item for cofinally many $\alpha$'s,
\[
 \algebra{B}_{\alpha+1}= \algebra{B}_\alpha* \dot{\asilver}.
\]
In this case, put $\dot{Y}_{\alpha+1} =
\dot{Y}_{\alpha}$ and $\dot{Z}_{\alpha+1} =
\dot{Z}_\alpha$.
\item for cofinally many $\alpha$'s, $\algebra{B}_{\alpha + 1}=
\algebra{B}_\alpha * \dot{\miller}$ and
\[
\dot{Y}_{\alpha +1}= \dot{Y}_\alpha \cup \{
\dot{y}_T: T \in \miller \},
\]
where $\dot{y}_T$ is a name for a Miller real over
$\model{N}^{\algebra{B}_\alpha}$ through $T \in \model{N}^{\algebra{B}_\alpha}$,
\item for cofinally many $\alpha$'s, $\algebra{B}_{\alpha + 1}=
\algebra{B}_\alpha * \dot{\random}$ and
\[
\dot{Z}_{\alpha +1}= \dot{Z}_\alpha \cup \{
\dot{z}_T: T \in \random \},
\] 
and $z_T$ is a name for a random real through the positive measure tree $T \in \model{N}^{\algebra{B}_\alpha}$.
\item for cofinally many $\alpha$'s we collapse $\alpha$ to $\omega$, i.e., $\algebra{B}_{\alpha + 1}=
\algebra{B}_\alpha * \textbf{Coll}(\omega, \alpha)$, and we let
$\dot{Y}_{\alpha +1}= \dot{Y}_\alpha$ and $\dot{Z}_{\alpha+1} =
\dot{Z}_\alpha$;
\end{enumerate}
 \item Finally, for any limit ordinal $\lambda$, $\dot{Y}_\lambda = \bigcup_{\alpha < \lambda} \dot{Y}_\alpha$, $\dot{Z}_\lambda = \bigcup_{\alpha < \lambda} \dot{Z}_\alpha$ and
$\algebra{B}_\lambda = \lim_{\alpha < \lambda} \algebra{B}_\alpha$.
\end{itemize}

\begin{remark}
Note that $\algebra{B}_\kappa$ is a direct limit of complete Boolean algebras of size $< \kappa$ collapsing $\kappa$ to $\omega_1$ and it is therefore trivially $\kappa$-cc.
At this point the reader could object that this algebra is nothing more than the Levy collapse, and hence the extension we get is the same as Solovay's. How can we then separate regularity properties? The point is that, even if we get the same forcing-extension obtained by Solovay, what we do is to look at a different inner model; after collapsing the inaccessible to $\omega_1$, we pick the inner model $\model{L}(\omega^\omega,Z,Y)$. So this method should be viewed as a technique for choosing the ``suitable'' inner model of Solovay's extension to obtain the required separation of regularity properties. For further observations about that, we refer the reader to the last section, questions 1, 3 and 6.

\end{remark}

The proof of the main theorem splits into the following lemmata.

\begin{lemma}
 Let $G$ be $\algebraQ_\kappa$-generic over $\model{N}$. Then
\[
 \model{N}[G] \models \text{``every set of reals in $\model{L}(\baire,Y,Z)$ is Silver measurable''}.
\]
\end{lemma}
 \begin{proof}
Fix arbitrarily $X \subseteq \cantor$, $\Phi$ and $r \in \baire$ such that $X=\{ x \in \cantor: \Phi(x,r,Y,Z) \}$. \footnote{Note that the argument works even if we start with $r \in \text{On}^\omega$. Hence, a similar proof actually holds for $X \in \text{HOD}(\text{On}^\omega, \{ Y \}, \{ Z \})$, as mentioned before. Furthermore, we remark that $\Phi$ will have (suppressed) ordinal parameters.}
Let $\alpha < \kappa$ be such that $r \in \model{V}[G \restric \alpha+1]$ and $\algebraQ_{\alpha+1}= \algebraQ_\alpha * \dot{\poset{VT}}$.
Note that, because of the first point of the construction above,
\begin{equation*} \label{eq4}
 \model{N}[G \restric \alpha+1] \models \text{``$\algebraQ_\kappa / G \restric \alpha+1$ is $(\silver, \dot{Y}, \dot Z)$-homogeneous''}.
\end{equation*}
Let $\model{N}^*= \model{N}[G \restric {\alpha+1}]$, $\algebraQ^* = \algebraQ_\kappa / G \restric {\alpha+1}$ and $H$ be the tail of the generic filter $G$, i.e., $H$ is $\algebraQ^*$-generic over $\model{N}^*$ and $\model{N}^*[H]= \model{N}[G]$.
Since the parameter $r$ has been ``absorbed'' in the ground model, for notational simplicity, from now on we will hide it, without indicating it explicitly within
the formula.
The next step will be to prove the Solovay property for $\Phi$ over Silver reals, which is the content of the next observation.
\begin{fact}
Let $\algebra{B}^*$ be $(\silver, \dot{Y},\dot{Z})$-homogeneous Boolean algebra, $\Phi(x,y,z)$ be a formula with only parameters in the ground model
and $Y, Z$ as parameters, and $\dot{x}$ be a name for a Silver real. Then $\Vert\Phi(\dot{x}, \dot{Y}, \dot{Z})\Vert_{\algebra{B}^*} \in \algebra{B}^*_{\dot{x}}$.
\end{fact}
\begin{proof}[Sketch of the proof] The proof is pretty standard and we give a sketch of it for completeness. To reach a contradiction, assume
$\Vert\Phi(\dot{x}, \dot{Y}, \dot{Z})\Vert_{\algebra{B}^*} \notin \algebra{B}^*_{\dot{x}}$. Let $\algebra{A}$ be the complete Boolean algebra generated by
$\algebra{B}^*_{\dot{x}} \cup \Vert\Phi(\dot{x}, \dot Y, \dot{Z})\Vert_{\algebra{B}^*}$. It is well-known that there exists $\rho: \algebra{A} \rightarrow \algebra{A}$
automorphism such that
$\rho(\Vert\Phi(\dot{x}, \dot Y, \dot Z)\Vert_{\algebra{B}^*}) \neq \Vert\Phi(\dot{x}, \dot Y, \dot Z)\Vert_{\algebra{B}^*}$ and $\rho$ is the identity over $\algebra{B}^*_{\dot{x}}$.
By $(\silver, \dot{Y}, \dot Z)$-homogeneity, there exists $\phi: \algebra{B}^* \rightarrow \algebra{B}^*$ automorphism extending $\rho$ such that
$\force_{\algebra{B}^*} \text{`` }\phi(\dot{Y})= \dot{Y} \text{ and } \phi(\dot{Z})= \dot{Z} \text{ ''}$ . Hence, the following equalities yield a contradiction:
\[
\begin{split}
 \rho(\Vert \Phi(\dot{x},\dot Y, \dot Z)\Vert_{\algebra{B}^*}) &= \phi(\Vert\Phi(\dot{x}, \dot Y, \dot Z)\Vert_{\algebra{B}^*}) \\
 													   &=\Vert\Phi(\phi(\dot{x}), \phi(\dot Y), \phi(\dot Z))\Vert_{\algebra{B}^*} \\
 													   &= \Vert\Phi(\dot{x}, \dot Y, \dot Z)\Vert_{\algebra{B}^*}.
\end{split}
\]
\end{proof}

\begin{fact} \label{fact:silver-char}
Let $\dot x$ be a $\algebra{B}^*$-name such that $\force_{\algebra{B}^*} \text{``$\dot x$ is a Silver real over $\model{N}^*$''}$,
and assume for every $b \in B(\silver)$, $\Vert \dot x \in b\Vert_{\algebra{B}^*} \neq \0$. Then there exists an isomorphism
\[
f: B(\silver) \rightarrow \algebra{B}^*_{\dot x}, \text{ such that } \force_{\algebra{B}^*} f(\dot v)= \dot x,
\]
where $\dot v$ is the canonical name for the Silver real.

\emph{(Hint: choose $f(b)= \Vert \dot{x} \in b \Vert_{\algebra{B}^*}$, for every $b \in B(\silver)$)}.
\end{fact}

Now let $\dot{x}_0$ be a name for a Silver real and assume $A=\Vert \Phi(\dot x_0,\dot{Y}, \dot Z) \Vert_{\algebra{B}^*} \neq \0$.
Hence, because of \ref{fact:silver-char}, together with $(\silver,\dot{Y}, \dot Z)$-homogeneity, one can consider $b=f^{-1}[A]$, with $b \in \algebra{B}(\silver)$.
The next observation is simply a version of Solovay's lemma, stated
for Silver reals in place of Cohen reals (for a proof, see \cite{BJ95}, lemma 9.8.5).
\begin{fact}
Suppose $\model{N}^*[H] \models \text{``$x$ is a Silver real over $\model{N}^*$''}$. Then
\[
 \model{N}^*[H] \models \text{``}  x \in b \ifif \Phi(x,Y,Z) \text{''}.
\]
\end{fact}

By remark \ref{remark-amoeba-absoluteness} and corollary \ref{absolute-amoeba}, one can pick a Silver tree $T \in \model{N}^*[H]$ such that $[T] \subseteq b$ and every $x \in [T]$ is Silver over $\model{N}^*$.
Hence, one obtains
\[
 \model{N}^*[H] \models \forall x \in [T] ( \Phi(x,Y,Z)),
\]
which precisely means $\model{N}[G] \models [T] \subseteq X$.

It is left to show the case $\Vert \Phi(\dot x, \dot Y, \dot Z)\Vert_{\algebraQ^*}=\0$.
In this case, $\Vert \neg \Phi(\dot x, \dot Y ,\dot Z)\Vert_{\algebraQ^*} \neq \0$ and then, arguing in the same way,
one gets a Silver tree $T \in \model{N}^*[H]$ such that $\model{N}^*[H] \models \forall x \in [T] (\neg \Phi (x,Y,Z))$, and therefore
\[
 \model{N}[G] \models [T] \cap X = \emptyset.
\]
\end{proof}

\begin{lemma} \label{lemma: Znotlebesgue}
Let $G$ be a $\algebra{B}_\kappa$-generic filter over
$\model{N}$. Then
\[
\model{N}[G] \models \text{`` $Z$ is not Lebesgue measurable ''}.
\]
\end{lemma}

\begin{proof}
In $\model{N}[G]$, we aim at showing that for every tree $S$ with positive measure,
both
\[
Z \cap [S] \neq \emptyset \text{ and } [S]
\nsubseteq Z.
\]
Let $\dot{S}$ be a
$\algebra{B}_\kappa$-name. There is $\alpha < \kappa$
such that $\dot{S}$ is a $\algebra{B}_\alpha$-name,
$\algebra{B}_{\alpha +1}=\algebra{B}_\alpha
* \dot{\random}$ and
$\dot{Z}_{\alpha+1}=\dot{Z}_\alpha \cup \{
\dot{z}_T: T \in \random \}$. Consider $\dot{z}_{S}$ name for a random
real over $\model{N}[G \restric
\alpha+1]$ such that $\model{N}[G] \models \dot{z}_S^G \in [S]$. Thus,
\[
\model{N}[G] \models \dot{z}_{S}^G \in Z \cap [S].
\]
On the other hand, there is also $\gamma < \kappa$, such that
$\dot{S}$ is a $\algebra{B}_\gamma$-name,
$\algebra{B}_{\gamma+1}= \algebra{B}_{\gamma}* \textbf{Coll}(\omega,\alpha)$
and $\dot{Z}_{\gamma+1}=\dot{Z}_\gamma$. Let
$\dot{g}$ be a name for a random real
over $\model{N}[G \restric \gamma+1]$ (added by the Levy collapse) such that
$\model{N}[G] \models \dot{g}^G \in [S]$. Obviously, $\model{N}[G] \models
\dot{g}^G \notin Z_{\gamma}$, since it is added at
stage $\gamma+1$, and thus,
\[
\model{N}[G] \models \dot{g}^{G} \in [S] \setminus
\dot{Z}_{\gamma+1}^{G},
\]
since
$\dot{Z}_{\gamma+1}=\dot{Z}_{\gamma}$.
It is left to show that $\model{N}[G] \models \dot{g}^{G}
\notin  Z \setminus
\dot{Z}_{\gamma+1}^{G}$. This follows from the following result.
\begin{lemma} \label{lemma:Z-unreachable}
For every $\gamma, \beta < \kappa$, $\gamma < \beta$, and $\dot{x} \in \dot{Z}_\beta \setminus \dot{Z}_\gamma$, one has
\[
\model{N}[G] \models \text{`` }\dot{x}^G \text{ is unreacheable over } \model{N}[G \restric \gamma+1]\text{ ''}.
\]
\end{lemma}

\begin{proof}[Proof of lemma \ref{lemma:Z-unreachable}] The proof is by induction on $\beta < \kappa$.

\textsc{Limit Case}: if $\beta$ is limit, then the result is completely trivial, since, given $\dot{x} \in \dot{Z}_\beta \setminus \dot{Z}_\gamma$, it follows that there exists
$\beta' < \beta$ such that $\dot{x} \in \dot{Z}_{\beta'} \setminus \dot{Z}_\gamma$, and so one can simply apply the inductive hypothesis for $\beta'$ in order to get
$\model{N}[G] \models \text{``}\dot{x}^G \text{ is unreachable over } \model{N}[G \restric \gamma+1]\text{''}$.

\textsc{Successor Case}: $\beta= \gamma+1$, i.e., $\dot{x} \in \dot{Z}_{\gamma+1} \setminus \dot{Z}_\gamma$. Two cases are possible:

\textsc{Subcase 1}: $\dot{Z}_{\gamma+1}= \dot{Z}_\gamma \cup \{\dot{z}_S: S \in \random \}$. In this case $\dot{x}$ has to be a random real over $\model{N}[G \restric \gamma+1]$ and therefore unreachable over it, because of remark \ref{remark:random-unreachable}.

\textsc{Subcase 2}: $\dot{Z}_{\gamma+1}= \dot{Z}_\gamma \cup \{ \phi^j(\dot{z}), \phi^{-j}(\dot{z})   : \dot{z} \in \dot{Z}_\gamma, j \in \omega \}$, where $\gamma = \alpha_{\eta}$ and $\phi= \phi_{\eta}$, for $\eta > 0$. First note that we have the following result, which is analog to lemma 3.4 in \cite{JR93}.

\begin{fact} \label{lemma:preserve-unreachable}
Let $\eta > 0$ be a successor ordinal.  Let $\algebra{B}', \algebra{B}'' \lessdot \algebra{B}_{\alpha_\eta}$ and $\dot{x} \in \model{N}^{\algebra{B}_{\alpha_\eta}} \cap \cantor$
such that
\[
 \force_{\algebra{B}_{\alpha_\eta}} \text{`` $\dot{x}$ is unreachable over both $\model{N}^{\algebra{B}'}$ and $\model{N}^{\algebra{B}''}$''},
\]
and $\psi: \algebra{B}' \rightarrow \algebra{B}''$ isomorphism.

Then, for every $j \in \omega$,
\[
 \force_{\algebra{B}_{\alpha_\eta+1}} \text{`` $\phi^j_\eta (\dot{x})$ and $\phi^{-j}_\eta (\dot{x})$ are unreachable over $\model{N}^{\algebra{B}_{\alpha_\eta}}$''}.
\]
where $\algebra{B}_{\alpha_\eta+1}=\omega$-$\amal(\algebra{B}_{\alpha_\eta}, \psi)$, and $\phi_\eta$ is the automorphism extending $\psi$, generated
by the amalgamation.
\end{fact}
\begin{proof}
The proof simply consists of a recursive application of lemma \ref{lemma:preserving-unreachability}. For an analogous case, one can see the proof of lemma 3.4 in \cite{JR93}.
\end{proof}

\begin{corollary} \label{corollary:Z}
Let $\algebra{B}_{\alpha_0} \lessdot \algebra{B}', \algebra{B}'' \lessdot \algebra{B}_{\alpha_1}$
such that
\[
 \force_{\algebra{B}_{\alpha_0}} \text{`` $(\algebra{B}': \algebra{B}_{\alpha_0}) \approx (\algebra{B}'': \algebra{B}_{\alpha_0}) \approx B^+(\silver)$''}
\]
and $\phi_0: \algebra{B}' \rightarrow \algebra{B}''$ isomorphism such that
$\phi_0 \restric \algebra{B}_{\alpha_0} = \emph{Id}_{\algebra{B}_{\alpha_0}}$. Then for every $\dot{x} \in \model{N}^{\algebra{B}_{\alpha_1}} \cap \cantor$ such that
$\force_{\algebra{B}_{\alpha_1}} \text{``$\dot{x}$ is unreachable over $\model{N}^{\algebra{B}_{\alpha_0}}$''} $, one has, for every $j \in \omega$,
\[
 \force_{\algebra{B}_{\alpha_1+1}} \text{`` $\phi^j_1 (\dot{x})$ and $\phi^{-j}_1 (\dot{x})$ are unreachable over $\model{N}^{\algebra{B}_{\alpha_1}}$''}.
\]
(As usual, $\algebra{B}_{\alpha_{1}+1}$ is the $\omega$-$\amal(\algebra{B}_{\alpha_1},\phi_0)$, and $\phi_1 \supseteq \phi_0$ the automorphism of $\algebra{B}_{\alpha_1+1}$ generated by the amalgamation.)
\end{corollary}
\begin{proof}
First, note that $\algebra{B}_{\alpha_0}$ forces both $(\algebra{B}':\algebra{B}_{\alpha_0}) \approx (\algebra{B}'':\algebra{B}_{\alpha_0}) \approx B^+(\silver)$, and then, by lemma \ref{lemma:silver-preserves-unreachable} and remark \ref{remark-silver}, we obtain
\[
 \force_{\algebra{B}_{\alpha_1}} \text{`` $\dot{x}$ is unreachable over both $\model{N}^{\algebra{B}_{\alpha_0}*(\algebra{B}': \algebra{B}_{\alpha_0})}$ and
$\model{N}^{\algebra{B}_{\alpha_0}*(\algebra{B}'': \algebra{B}_{\alpha_0})}$''}.
\]
To finish the proof, one can then apply fact \ref{lemma:preserve-unreachable}, for $\eta=1$.
\end{proof}
Going back to the proof of Subcase 2, we have two cases:
\begin{itemize}
 \item $\eta=1$, and so $\gamma= \alpha_1$: in such a case either $\dot{x}= \phi^j_1(\dot{z})$ or $\dot{x}= \phi^{-j}_1(\dot{z})$, for some $\dot{z} \in \dot{Z}_{\alpha_1}$ and $j \in \omega$. Hence, by inductive
hypothesis, we have $\force_{\algebra{B}_{\alpha_1}} \text{`` $\dot{z}$ is unreachable over $\model{N}^{\algebra{B}_{\alpha_0}}$''}$, and therefore, by corollary
\ref{corollary:Z}, we obtain
\[
\force_{\algebra{B}_{\alpha_1+1}} \text{`` $\dot{x}$ is unreachable over $\model{N}^{\algebra{B}_{\alpha_1}}$''};
\]
\item $\eta > 1$ and so $\gamma= \alpha_{\eta}$: in such a case we do not have to use corollary \ref{corollary:Z}, but fact \ref{lemma:preserve-unreachable} is sufficient; in fact, in this case, we do not have the ``intrusion'' of the two copies of the Silver$^+$ algebra in the amalgamation. More precisely, if $\dot x = \phi_{\eta}^{j}(\dot z)$ for some $\dot z \in Z_{\alpha_{\eta}}$, then one obtains
\[
\force_{\algebra{B}_{\alpha_{\eta}+1}} \text{`` $\dot{x}$ is unreachable over $\model{N}^{\algebra{B}_{\alpha_{\eta}}}$''},
\]
since by inductive hypothesis $\force_{\algebra{B}_{\alpha_{\eta}}} \text{`` $\dot{x}$ is unreachable over  $\model{N}^{\algebra{B}_{\alpha_{\eta-1}}}$ ''}.$
\end{itemize}
\end{proof}
By our previous comments, this concludes the proof to show $Z$ not being Lebesgue measurable.
\end{proof}
\begin{lemma} \label{lemma: YnotMiller}
Let $G$ be a $\algebra{B}_\kappa$-generic filter over
$\model{N}$. Then
\[
\model{N}[G] \models \text{`` $Y$ is not Miller measurable ''}.
\]
\end{lemma}

\begin{proof}
The proof is analogous to the one just given for showing $Z$ not being Lebesgue measurable. Here, instead of using random reals, we use Miller reals, and instead of using the unreachability, we use the unboundedness. In fact an analogous of fact \ref{lemma:preserve-unreachable} and corollary \ref{corollary:Z} can be proven if one replaces the word ``unreachable'' with ``unbounded'' (see \cite{JR93}, lemma 3.4 and lemma 6.1). One can then continue with a similar proof, by using the fact that Miller reals are unbounded over the ground model and that Silver forcing is $\baire$-bounding.
\end{proof}

Hence, if one considers the inner model $\model{L}(\real, Y,Z)$ of $\model{N}[G]$, one obtains
\begin{eqnarray*}
\model{L}(\real,Y,Z)^{\model{N}[G]} &\models& \vm \land \neg \lm \land \neg \mm \land \\
											&& \forall x \in \baire (\omega_1^{\model{L}[x]} < \omega_1),
\end{eqnarray*}

\begin{remark}
Note that any comeager set contains the branches through a Miller tree, and therefore $\bp \Rightarrow \mm$, by lemma \ref{lemma:theta-measurability}. Hence, $\bp$ fails in our model, without displaying a concrete counterexample. On the contrary, note that our method does not permit to construct a unique set $Y$ which is simultaneously non-Miller measurable and non-Lebesgue measurable. In fact, on the one hand random reals are unreachable but not unbounded, whereas on the other hand Miller reals are unbounded but not unreachable.
\end{remark}

\textbf{Final acknowledgement.} I really would like to thank the anonymous referee for the suggestions which have definitely made the exposition much clearer. In particular, one of his/her observations has led to question 1 of the coming section.

\section{Concluding remarks and open questions} \label{conclusion}
We conclude with some questions which we consider noteworthy and for which further developments are expected.
\begin{itemize}
\item[(Q1)] It would be interesting to understand the behaviour of the inaccessible $\kappa$ if we do not explicitely collapse it along the construction. In the model presented by Shelah in \cite{Sh85}, if we start from $\model{L}$ as ground model and $\kappa$ being the least inaccessible, we know that $\kappa$ collapses to $\omega_1$; in fact, $\omega_1$ has to be inaccessible by reals, as the latter is implied by $\SSigma^1_3(\textsc{Lebesgue})$. Nevertheless, in our case, we know that projective Silver measurability has the consistency strength of ZFC, so we cannot use such an \emph{indirect} argument. We conjecture that the algebra $\algebra{B}_\kappa$ will anyway collapse the inaccessible to $\omega_1$, but we currently do not have a precise proof of that. 

\item[(Q2)] In his PhD dissertation \cite{S10} written under the supervision of Sy Friedman, David Schrittesser improved the amalgamation-method in order to get a projective version of Shelah's result, where all \emph{projective} sets are Lebesgue measurable and there exists a \emph{projective} set without Baire property. So a natural question is whether such a method can be useful to obtain the projective version of the separation between Silver-, Miller- and Lebesgue- measurability that we presented in this paper, and as usual requiring $\omega_1$ to be inaccessible by reals.
\item[(Q3)] The same method might be done to separate other two regularity properties: in our case, we used the fact that a random real is unreachable over the ground model, and that the unreachability is somehow preserved by amalgamation and Silver forcing. Obviously, in other cases, the trick will be to find the right property of the generic real still preserved by the amalgamation and simultaneously ``respected'' by the other forcing; for example, if we want to get $\sm \wedge \neg \vm$ one should find a particular feature of the Silver real which is preserved by amalgamation and Sacks forcing in the sense of fact \ref{lemma:preserve-unreachable} and corollary \ref{corollary:Z}.
\item[(Q4)]  The previous results and observations point out that the nature of the generic reals, and more generally of the forcing notions, is strictly related to the behaviour of the regularity properties associated. Hence, a more general and intriguing question could be to understand whether some specific relation between two tree-like forcings $\poset{P}$, $\poset{Q}$ reflects on the relation between $\textbf{all}(\poset{P}\textsc{-measurability})$ and $\textbf{all}(\poset{Q}\textsc{-measurability})$, where $\poset{P}$-measurability denotes the notion of regularity associated with $\poset{P}$ (see \cite{Ik10} and \cite{K12}, chapter 2). For example, we know that $\cohen \lessdot \dominating$ (where the latter is the Hechler forcing) and that $\textbf{all}(\textsc{Hechler}) \Rightarrow \bp$; Can one obtain a more general fact asserting that if $\poset{P} \lessdot \poset{Q}$ then $\textbf{all}(\poset{Q}\textsc{-measurability})$ implies $\textbf{all}(\poset{P}\textsc{-measurability})$?
\item[(Q5)] Since any comeager set contains the branches through a Miller tree, it follows that $\bp \Rightarrow \mm$. Then it comes rather natural to ask whether or not such an implication can be reversed. If we want to apply the method presented in this paper to give a negative response, we should find a property which is preserved via Miller extension, and satisfied by Cohen reals. Note that the unreachability cannot help for that, since Miller reals themselves are not unreachable.
\item[(Q6)] Starting from $\model{L}$ and using Shelah's machinery to obtain $\bp$ without inaccessible, we get a model where even $\DDelta^1_2(\textsc{Lebesgue})$ fails, since by sweetness no random reals are added. Furthermore, in such a model we obviously have $\omega_1^{\model{L}}=\omega_1$. What about a model for $\bp \land \neg \lm$ but in which $\omega_1$ is inaccessible by reals? The answer is far from trivial, since Shelah's method seems to have several difficulties in that case; in fact, one should find a property of the random real which is preserved by Cohen extension (and simultaneously by the amalgamation), which appears really hard to obtain. So it seems that a completely different method should be used, probably even another way to construct homogeneous algebras.

\end{itemize}

\end{document}